\documentclass[12]{article}
\usepackage{fleqn}
\usepackage{graphicx}
\usepackage{hyperref}
\usepackage{amsthm}
\usepackage{amsfonts}
\usepackage{amsmath}
\usepackage{amssymb}
\usepackage{verbatim}
\usepackage{latexsym}
\theoremstyle{definition}
\newtheorem{defi}{Definition}

\theoremstyle{plain}
\newtheorem{thm}[defi]{Theorem}
\newtheorem{lem}[defi]{Lemma}
\newtheorem{cor}[defi]{Corollary}

\begin{document}
\Large
\begin{center}
{\bf Free Cyclic Submodules and Non-Unimodular Vectors}
\end{center}
\large
\vspace*{-0.0cm}
\begin{center}
Joanne L. Hall$^{1,2}$ and Metod Saniga$^{2}$
\end{center}
\vspace*{-.4cm} \normalsize
\begin{center}

$^{1}$School of Mathematical and Geospatial Sciences, RMIT University\\ GPO Box 2476, Melbourne 3001\\ Australia\\
(joanne.hall@rmit.edu.au)

\vspace*{.05cm} and

\vspace*{.05cm}

$^{2}$Astronomical Institute, Slovak Academy of Sciences\\
SK-05960 Tatransk\' a Lomnica\\ Slovak Republic\\
(msaniga@astro.sk)

\end{center}

\vspace*{-.2cm} \noindent \hrulefill

\vspace*{.1cm} \noindent {\bf Abstract}

\noindent Given a finite associative ring with unity, $R$, and its two-dimensional left module, $^{2}\!R$, the following two problems are addressed: 1) the existence of vectors of  $^{2}\!R$ that do not belong to any free cyclic submodule (FCS) generated by a unimodular vector and 2) conditions under which such (non-unimodular) vectors generate FCSs. The main result is that for a non-unimodular vector to generate an FCS of $^2\!R$, $R$ must have at least two maximal right ideals of which at least one is non-principal.
\\ \\
%{\bf MSC Codes:} 51Exx, 81R99\\
%{\bf PACS Numbers:} 02.10.Ox, 02.40.Dr, 03.65.Ca\\
{\bf Keywords:}  Finite Unital Rings -- Free Cyclic Submodules -- Non-Unimodular Vectors

\vspace*{-.1cm} \noindent \hrulefill

\vspace*{.3cm}
%\large
\section{Introduction}
Projective geometries over finite associative rings with unity have recently found important applications in coding theory (see, e.\,g., \cite{HL09}) and quantum information theory (see, e.\,g., \cite{PSK06, HS08, SPP08}).
When constructing a geometry over a ring, a majority of authors consider as points of such a geometry only free cyclic submodules (FCSs) generated by  unimodular vectors \cite{veld95}, whilst some authors  consider
all cyclic  submodules \cite{BGS95, Faure04}. It has recently been shown \cite{HS09} that there exists rings for which some vectors  of the submodule are not  contained in an FCS generated by a unimodular vector.  These vectors have been called outliers.  Even more interesting is that some outliers themselves generate FCSs. A geometry may be constructed using all FCSs.

Analysing all finite associative rings with unity up to order 31 inclusive, only several rings are found to feature outliers. Out of these, only few non-commutative rings exhibit FCSs comprising solely non-unimodular vectors \cite{San11}; the smallest example being the ring of ternions over the Galois field of order two \cite{HS09}. These examples motivated  a more systematic and general treatment of the questions of the existence of outliers and FCSs generated by them. The outcomes of our explorations are not only interesting on their own, but they can also have interesting physical bearings (like, e.\,g., those proposed in \cite{SP10}).

\section{Definitions and Preliminaries \label{defi}}
All rings considered are finite, associative and with unity (multiplicative identity). It is well known that in such a ring, $R$, an element is either a unit or a (two-sided) zero-divisor (see, for example, \cite[\S 2.1]{Rag69}); in what follows the group of units of $R$ is denoted by $R^*$ and the set of zero divisors by $R\setminus R^*$.  $1$ is the unity element of $R$ and the symbol $\subset $ stands for strict inclusion.

\begin{defi}Let $\langle R, \cdot, +\rangle $ be a ring.  A \emph{left ideal}, $I_l$, is a subgroup of $\langle R,+\rangle$ such that $rx\in I_l$ for all $r\in R$ and $x\in I_l$.  A \emph{right ideal}, $I_r$, is a subgroup of $\langle R,+\rangle$ such that $xr\in I_r$ for all $x\in I_r$ and $r\in R$.  An ideal is \emph{principal} if it is generated by a single element of $R$. For $a\in R$ the  \emph{principal left ideal} generated by $a$ is  $Ra$, and a \emph{principal right ideal} generated by $a$ is $aR$.
\end{defi}
For further background on rings see, for example, \cite{Lam01}.
\begin{defi}\cite[Defi 2.9]{veld95}\cite[p.\,16]{HP73} Let $S\subseteq R$. The left (right) \emph{annihilator} of $S$, denoted $^\perp\!S$ ($S^\perp$), is defined as:
\begin{eqnarray*}
 ^\perp\!S & = & \{x\in R:xa=0, \forall a\in S\},\\
S^\perp & = & \{x\in R:ax=0, \forall a\in S\}.
\end{eqnarray*}
\end{defi}
For sets containing a single element, the notation is simplified $^\perp\!\{a\}:=\;^\perp\!a$.

\begin{lem}\label{lem:perpinclusion} \label{lem:perp}Let $P,S\subseteq R$, then $^\perp\!S$ is a left ideal, $S^\perp$ is a right ideal and
 \begin{equation*}
  ^\perp\!P\cap\;^\perp\!S=\;^\perp(P\cup S).
 \end{equation*}
\end{lem}
Definitions below are given for left modules, the mirrored definitions can be given for right modules.

\begin{defi}\cite{veld95}
 Let $R$ be a ring with unity, and  $^2\!R$ be a left module over $R$, and let
\begin{equation}
aR+bR=\{ax+by:x,y\in R\}.
\end{equation}
  $(a,b)\in\; ^2\!R$ is  unimodular if $aR+bR=R$.
  \end{defi}
  Note that $aR$ and $bR$ are principal right ideals of $R$.  The following Lemma provides an alternate definition of unimodular.
\begin{lem}\cite{veld95}
 Let $R$ be a ring with unity, and  $^2\!R$ be a left module over $R$. $(a,b)\in\; ^2\!R$ is  unimodular if and only if  there exists $x,y\in R$ such that $ax+by=1$.
\end{lem}

\begin{defi} $R(a,b)$ is a \emph{cyclic} subset of $^2\!R$ generated by $(a,b)$:
\begin{equation*}
 R(a,b)=\{(\alpha a,\alpha b):\alpha\in R\}.
\end{equation*}
If $(\alpha a,\alpha b)=(0,0)$ only when $\alpha=0$, then $R(a,b)$ is a \emph{free cyclic submodule}.
\end{defi}
Reworking the definition of a free cyclic submodule using annihilators leads to the obvious lemma:

\begin{lem}
 Let $R$ be a finite associative ring with unity. $R(a,b)$ is a free cyclic submodule of $^2\!R$ if and only if
\begin{equation}\label{eqn:perp}
 ^\perp a\cap\;^\perp b=\{0\}.
\end{equation}
\end{lem}
\begin{proof}
  Let $(a,b)$ be the generating vector, then by definition $R(a,b)$ is free if  $\alpha(a,b)=(0,0)$ only if $\alpha=0$.  This is equivalent to equation (\ref{eqn:perp}).
\end{proof}
%We take the set of  free cyclic submodules of $^2\!R$ as the points of the projective line over $R$.  See \todo{citation} for more on the geometry.

\begin{lem}\cite[\S 1]{veld95}\label{lem:veld}
 Let $(a,b)$ be a unimodular vector in $^2\!R$, then $R(a,b)$ is a free cyclic submodule.
\end{lem}

Unimodular vectors have other useful properties \cite{veld95}, and for many rings all free cyclic submodules are generated by unimodular vectors.  However this is not always the case. Corollaries  \ref{cor:local} and \ref{cor:principalideal} show two  classes of rings for which all free cyclic submodules are generated by unimodular vectors. The ring of ternions \cite{HS09} is an example where some free cyclic submodules are generated by non-unimodular vectors.

\begin{defi}
 An \emph{outlier} is a vector which is not contained in any  free cyclic submodule generated by a unimodular vector.
\end{defi}

The aim of this research is to get some insight into which rings contain outliers, and, more specifically, which rings contain outliers that generate free cyclic submodules.  This question is of interest for general $^n\!R$, but we only treat the simplified case of $^2\!R$ where $R$ is a finite associative ring.

\section{Results}
\subsection{Unimodular vectors}
We begin by collecting some important facts about unimodular vectors.

If $a\in R^*$, then $aR=R$, hence for all $b\in R$, $aR+ bR=R$. Thus any vector containing a unit as an entry is a unimodular vector.
Unimodular vectors may be divided into two types:
\begin{itemize}
 \item Type I: vectors which contain at least one entry which is a unit;
 \item Type II: vectors which contain no  entries that are units.
\end{itemize}
\begin{thm}\label{thm:2maxideals}
 Let $R$ be a finite associative ring with unity. If $a,b\in R\setminus R^*$, then  $(a,b)$ is a  unimodular vector in $^2\!R$ if and only if there exist maximal right ideals, $I_1$, $I_2$, such that $a\in I_1\setminus I_2$ and $b\in I_2\setminus I_1$.
\end{thm}
\begin{proof}
$\Rightarrow$ Assume $(a,b)$ is unimodular.  Since $a,b\in R\setminus R^*$, $aR$ and $bR$ are right ideals strictly contained in $R$.  If there is some proper right ideal, $I$, that contains $a$ and $b$, then $aR+ bR\subseteq I\subset R$.  Hence, if $(a,b)$ is unimodular, then $a$ and $b$ cannot be contained in the same maximal ideal.

$\Leftarrow$ Assume $a\in I_1\setminus I_2$ and $b\in I_2\setminus I_1$, then $aR$ and $bR$ are right ideals for which $aR\subset I_1$ and $bR\subset I_2$.  $aR+ bR$ is a right ideal not contained in either $I_1$ or $I_2$.  Therefore $aR+ bR$ must be contained in a right ideal that contains both $I_1$ and $I_2$.  Since $I_1$ and $I_2$ are maximal, the only right ideal containing them both is $R$.  Hence $aR+ bR=R$, and $(a,b)$ is a unimodular vector of $^2\!R$.
\end{proof}
All type II unimodular vectors of $^2\!R$ conform to the conditions of Theorem \ref{thm:2maxideals}.

\begin{thm}\label{thm:alsounimodular} \label{thm:subset}
Let $R$ be a finite associative ring with unity.  $a,b,\alpha\in R$.
 \begin{enumerate}
 \item
 $R(\alpha a,\alpha b)\subseteq R(a,b)$.
 \item
 $R(\alpha a,\alpha b)=R(a,b)$ if and only if $\alpha\in R^*$.
 \item
 If $(a,b)$ is a unimodular vector in $^2\!R$, then $(\alpha a,\alpha b)$ is also unimodular if and only if $\alpha\in R^*$.\label{thm:4}
\end{enumerate}
\end{thm}
\begin{proof}
1, 2.  Let $R$ be a finite associative ring with unity. Then $R(\alpha a,\alpha b)\subseteq R(a,b)$.  If  $\alpha\in R^*$, then $R(\alpha a,\alpha b)= R(a,b)$.
3. \cite[Prop 2.1]{BH00}.
%
%$\Rightarrow$ $(a,b)$ is unimodular, thus by definition there exists some $x,y\in R$ such that $ax+by=1$.  Let $\alpha\in R^*$ then
%\begin{eqnarray*}
% ax+by & = & 1,\\
%\alpha(ax+by) & = & \alpha,\\
%\alpha(ax+by)\alpha^{-1} & = & 1,\\
%(\alpha a)(x\alpha^{-1})+(\alpha b)(y\alpha^{-1}) & = & 1.
%\end{eqnarray*}
%Hence $(\alpha a,\alpha b)$ is unimodular.
%
%$\Leftarrow$ Suppose $(\alpha a,\alpha b)$ is unimodular, then for some $x,y\in R$
%\begin{eqnarray*}
% \alpha ax+\alpha by & = & 1,\\
%\alpha (ax+by) & = & 1;
%\end{eqnarray*}
%hence, $\alpha\in R^*$.
\end{proof}

%\begin{thm}

%\end{thm}
%\begin{proof}
%$R$ is associative, thus $\beta(\alpha a)=(\beta\alpha) a$ for all $\beta,\alpha, a\in R$.  Thus $R(\alpha a,\alpha b)\subseteq R(a,b)$.  Let $\alpha\in R^*$, $\beta a=\beta\alpha^{-1}\alpha a$, thus $R(a,b) \subseteq R(\alpha a,\alpha b)$.

%For unimodular vectors this is a corollary of Theorem \ref{thm:alsounimodular}.
%\end{proof}

Theorems \ref{thm:2maxideals} and \ref{thm:subset} give criteria for checking for unimodular vectors of $^2\!R$.  More difficult is finding outliers.

\subsection{Outliers}

In the light of Theorem \ref{thm:subset} we can refine our notion of outlier.
\begin{defi}
Let $R$ be a finite associative ring with unity.  $(a,b)$ is an \emph{outlier} of $^2\!R$ if there does not exist $\alpha,c,d\in R$ such that $(a,b)=(\alpha c,\alpha d)$ and $(c,d)$ is unimodular.
\end{defi}

\begin{thm}\label{thm:outlier}
Let $R$ be a finite associative ring with unity.  $(a,b)$ is an outlier of $^2\!R$  if and only if there exists a  right ideal $I\subseteq R$, such that  $a,b\in I$ and
\begin{enumerate}
 \item  there are no principal right ideals which contain both $a$ and $b$;
\item  for all principal right ideals $\alpha R$ such that $a,b\in \alpha R$, then $aR+ bR\subset \alpha R$.
\end{enumerate}
\end{thm}
\begin{proof}
$\Rightarrow$ Theorem \ref{thm:2maxideals}
shows that if $(a,b)$ is an outlier of $^2\!R$ then $a$ and $b$ are contained in some   maximal right ideal.
Either $a,b\not\in \alpha R$ for some $\alpha \in R\setminus R^*$ (showing part 2) or there exists $\alpha,c,d$ such that $(\alpha c,\alpha d)=(a,b)$ only if $(c,d)$ is not unimodular.
Assume that there exists  $(\alpha c,\alpha d)=(a,b)$ with  $(c,d)$ not unimodular.  Let $C=\{c:\alpha c=a\}$ and $D=\{d:\alpha d=b\}$.  If   $\alpha x=\alpha y$ with $x\neq y$, then $\alpha(x-y)=0$, thus $(x-y)\in \alpha^\perp$. Since,
\begin{equation}\label{eqn:plusperp}
 C=\{c\}+ \alpha^\perp\quad\quad \mbox{and}\quad\quad D=\{d\}+ \alpha^\perp,
\end{equation}
one gets
\begin{align*}
 aR+ bR  & = \alpha cR+ \alpha d R \\
 & =\alpha(cR+ dR+ \alpha^\perp).
\end{align*}
Thus $aR+ bR =\alpha R$ if and only if there exists $c\in C$ and $d\in D$ such that $cR+ dR+ \alpha^\perp=R$.
let $CR=\{cr:c\in C,r\in R\}$ and $DR=\{dr:d\in D,r\in R\}$.   Since we can choose any $c\in C$ and $d\in D$, we require that
\begin{align*}
 CR+ DR+ \alpha^\perp & =R.
\end{align*}
From equation (\ref{eqn:plusperp}) it follows:
\begin{align*}
 CR+ DR+ \alpha^\perp & = CR+ DR.
\end{align*}
If $CR+ DR=R$, then there exists $c\in C$ and $d\in D$ and $x,y\in R$ such that $cx+dy=1$, implying that $(c,d)$ is a unimodular vector.  This contradicts that $(a,b)$ is an outlier, and hence we find that $aR+ bR\subset \alpha R$ (showing part 1).

$\Leftarrow$
1. If all right ideals that contain $a$ and $b$ are non-principal, then there does not exist $\alpha\in R\setminus R^*$ such that $a,b\in \alpha R$.  Hence $(a,b)$ is an outlier of $^2\!R$.

2. Let  $\alpha R$ be a principal right ideal for which $a,b\in \alpha R$. Then there exists $c,d$ such that $\alpha c=a$ and $\alpha d=b$.   If $aR+ bR\subset \alpha R$, then
\begin{align*}
 (\alpha c)R+ (\alpha d)R & \subset \alpha R,\\
\alpha(cR+ dR) & \subset \alpha R,\\
cR+ dR & \subset R,
\end{align*}
and $(c,d)$ is not unimodular.  This holds for all principal right ideals containing $a$ and $b$;  hence $(a,b)$ is an outlier.
\end{proof}

%\begin{cor}
 %Let $R$ be a finite associative, commutative ring with unity.  $(a,b)$ is an outlier of $^2\!R$ if and only if there exists a non-principal ideal which contains both $a$ and $b$, but no principal ideal which contains both $a$ and $b$.
%\end{cor}
If $R$ is commutative then $(a,b)$ is an outlier of the left module exactly when $(a,b)^T$ is an outlier of the right module.  In a non-commutative ring, the set of left outliers may be different to the set of right outliers (the smallest example is the ring of ternions of order 8 \cite{HS09}).  The set of outliers is dependent on the structure of the ideals of the ring.  If the left and right ideals of a ring have  different structures, then a left outlier may be right unimodular.

From Theorems \ref{thm:2maxideals} and \ref{thm:outlier} we get that:
\begin{lem}  If $a$ and $b$ are in some right ideal which is non principal,  not both in any principal right ideal, and not both in the same  maximal left ideal, then $(a,b)$ is an outlier and $(a,b)^T$ is  unimodular.
\end{lem}

\subsection{Free cyclic submodules and outliers}
  We have established that the structure of the ideals of a ring determines the set of outliers and unimodular vectors.  The Jacobson radical is an important ideal and  of crucial importance in the study of unimodular vectors.

\begin{defi}\cite[\S 4]{Lam01} For a finite ring $R$, the \emph{Jacobson radical}, $\textrm{rad}(R)$,  may be equivalently defined as:
\begin{itemize}
 \item  the intersection of all the maximal left ideals of $R$;
\item the largest left ideal  $J$ such that $1+j\in R^*$ for all $j\in J$.
\end{itemize}
\end{defi}
Note that the Jacobson radical is a left {\em and} right ideal.
%\begin{lem}\label{lem:subideal}
%Let $J$ be the Jacobson radical of $R$. Every free cyclic submodule $R(a,b)$ of $^2\!R$ contains
%at least $|J|$ vectors which have only the entries from $J$.
%\end{lem}
%\begin{proof}
%Let $j\in J$, then $ja\in J$ and $(ja,jb)\in\; ^2\!J$.  Since $R(a,b)$ is free, if $j\neq k$ then  $(ja,jb)\neq (ka,kb)$.   Hence there are at least $|J|$ vectors, including $(0,0)$ in $R(a,b)$ which contain only the elements of $J$.
%\end{proof}
\begin{defi}\cite[Defi 4.9]{Lam01} A one-sided or two-sided ideal, $I$, is \emph{nilpotent} of \emph{nilpotency} $m$ if $a_1.a_2\dots a_m=0$ for any set of elements $a_1,a_2,\dots,a_m \in I$.
\end{defi}

\begin{lem}\cite[Thm 4.12]{Lam01}\label{thm:Jnilpotent} Let $R$ be a finite associative ring.  Then $\textrm{rad}(R)$ is nilpotent.
\end{lem}

\begin{thm}\label{thm:1}
Let $R$ be a finite associative ring. Let $J \equiv \textrm{rad}(R)$.  Then no vector from $^n\!J$  generates a free cyclic submodule.
\end{thm}
\begin{proof}
From Lemma \ref{thm:Jnilpotent} it readily follows that
  $J$ has nilpotency $m$ for some $m\in \mathbb{N}$. Let $(a_1, a_2,\dots, a_n) \in \;^n\!J$, $x_1, x_2,\dots, x_{m-1} \in J$  and $\alpha = x_1.x_2\dots x_{m-1}$. Then  $(\alpha a_1, \alpha a_2,\dots, \alpha a_n)\\ = (0, 0,\dots, 0)$. Hence $R(a_1, a_2,\dots, a_n)$ is not a free cyclic submodule.
\end{proof}

\begin{defi}\cite[\S 19]{Lam01}
 A \emph{local} ring is an associative ring that has exactly one maximal left (and also right) ideal.
\end{defi}

As a side note we mention that geometries over local rings are called Hjelmslev geometries \cite[\S 9]{veld95}, and have applications in coding theory \cite{HL09}.

\begin{thm}\label{thm:local_outlier}
 Let $R$ be a local ring.  \begin{enumerate}
\item   No outliers of $^2\!R$ generate free cyclic submodules.
 \item $(a,b)$ is an outlier of $^2\!R$ if and only if $a\not\in bR$ and $b\not\in aR$.
 \end{enumerate}
\end{thm}
\begin{proof}
1. $R$ has exactly one maximal ideal, which is therefore the Jacobson radical,  $J$.  All ring elements not belonging to $J$ are units.  Hence any outlier of $^2\!R$ has both entries as elements of $J$. Theorem \ref{thm:1} shows that no vectors with both entries from $J$ can generate a free cyclic submodule.

2. $J$, the unique maximal ideal of $R$, cannot generate $R$.  By Theorem \ref{thm:2maxideals}, no unimodular vector of $^2\!R$ can contain elements of the same maximal ideal.  Since every element of $R$ is either a unit or an element of $J$, all unimodular vectors have a unit entry; all unimodular vectors are of type I. So, the outliers of $^2\!R$ are those vectors which are not contained in a free cyclic submodule of $^2\!R$  generated by $(1,x)$ or $(x,1)$, $x\in R$.  A vector which is not an outlier is of the form
\begin{equation*}
 (a,ax) \quad\quad\mbox{or}\quad\quad (ax,a),\;\mbox{for some}\;a,x\in R.
\end{equation*}
Thus outliers are those vectors which do not fit this form.  If $(a,b)$ is contained in a free cyclic submodule then there exists $x\in R$, such that $ax=b$ or $bx=a$.  Hence $b\in aR$ or $a\in bR$.  If $a\not\in bR$ and $b\not\in aR$, then $(a,b)$ is an outlier.
\end{proof}

\begin{cor}\label{cor:local}
 If $R$ is a finite local ring, then $R(a,b)$ is a free cyclic submodule if and only if $(a,b)$ is unimodular.
\end{cor}
This is a class of rings for which all free cyclic submodules are generated by unimodular vectors (the reverse of Lemma \ref{lem:veld}).
In particular, this means that Hjelmslev geometries (which are geometries over a local ring) cannot have non-unimodular points.

Next we look at another property of ideals which precludes the existence of non-unimodular free cyclic submodules.

\begin{lem}
\label{thm:neccesaryright}
Let $R$ be a finite associative ring with unity.  Let $a,b$ be elements of the same principal proper right ideal, $\alpha R$, then $R(a,b)$ is not a free cyclic submodule of $^2\!R$.
\end{lem}
\begin{proof}
$a=\alpha c$ and $b=\alpha d$.  Then
\begin{align*}
 ^\perp\! a & =\{x:xa=0\}\\
 & =  \{x:x\alpha c=0\}\\
 & \supseteq \{x:x\alpha=0\}\\
 & = \;^\perp\!\alpha.
\end{align*}
By the same logic $^\perp\! b\supseteq \;^\perp\!\alpha$.  Hence $^\perp\!a\cap\;^\perp\!b\supseteq \;^\perp\!\alpha\neq \{0\}$.
\end{proof}
Lemma \ref{thm:neccesaryright} then gives the following important result.
\begin{thm}\label{thm:nonprincipalideals} Let $R$ be a finite associative ring with unity.  If every right ideal is a principal ideal, then there are no free cyclic submodules of $^2\!R$ generated by non-unimodular vectors.
\end{thm}
This shows that a necessary condition for the existence of non-unimodular free cyclic submodules of $^2\!R$ is the presence of non-principal  right ideals.
%\begin{cor}Let $(a,b)$ be unimodular in $^2\!R$.  Then $(a,b)^T$ generates a free cyclic submodule of $R^2$.\todo{this only works if I have my lefts and rights correct}
%\end{cor}
\begin{cor}\label{cor:principalideal}
 Let $R$ be a principal ideal ring,  then $R(a,b)$ is a free cyclic submodule if and only if $(a,b)$ is unimodular.
\end{cor}
This is another class of rings (see Corollary \ref{cor:local}) for which all free cyclic submodules are generated by unimodular vectors.

When using associative rings,  free cyclic submodules are generated by either unimodular vectors or outliers.  If the assumption of associativity is removed, then this is no longer true.

\begin{lem}
 Let $R$ be a finite ring with unity and let  $(a,b)$ be a unimodular vector in $^2\!R$.  If there exists $\alpha$ such that $(\alpha a,\alpha b)$ is a non-unimodular vector  and $R(\alpha a,\alpha b)$ is a free cyclic submodule of $^2\!R$, then $R$ is non-associative.
\end{lem}
\begin{proof}
 Assume that $R$ is associative.  Then, by Theorem \ref{thm:4}, if $(\alpha a,\alpha b)$ is non-unimodular, then $\alpha\in R\setminus R^*$.  If $R(\alpha a,\alpha b)$ is free, then $R(\alpha a,\alpha b)=R(a,b)$.  Thus there exists $\beta\in R$ such that
\begin{equation*}
 (\beta\alpha a,\beta\alpha b)= (a,b),
\end{equation*}
under the assumption that $R$ is associative, this requires that $\beta\alpha=1$, contradicting that $\alpha\in R\setminus R^*$.

Hence if  there exists $\alpha$ such that $(\alpha a,\alpha b)$ is a non-unimodular vector  and $R(\alpha a,\alpha b)$ is a free cyclic submodule, then $R$ is non-associative.
\end{proof}

Examples have been calculated of non-associative rings of order $8$, where $(a,b)$ is unimodular, $R(\alpha a,\alpha b)$ is free and $R(\alpha a,\alpha b)\not\subseteq R(a,b)$.

\section{Conclusion and Further Directions}
For the existence of FCSs of $^2\!R$ that are generated by non-unimodular vectors  (``non-unimodular FCSs"), $R$ must have at least two maximal right ideals, at least one of which is non-principal.  This is a necessary condition.  Calculated examples \cite{San11} show that this condition is not sufficient; other properties of a ring are required to guarantee the presence of FCSs generated by non-unimodular vectors.

As already mentioned in the introduction, in our worked examples \cite{San11} non-unimodular FCSs have been only found for non-commutative rings. One would be tempted to conjecture that non-commutativity is essential in this respect. Yet, this is questionable because some rings feature non-unimodular FCSs in $^2\!R$, but not in $R^{2}$ (and {\it vice versa}). Hence, it is highly desirable to clarify to what extent the existence of non-unimodular FCSs depends on the non-commutativity of the ring; in particular, what is the smallest commutative ring featuring non-unimodular FCSs?

Further, in all analysed examples, a non-unimodular FCS was found to share with {\it any other} FCS at least one vector apart from $(0,0)$; is this true in general, or just a feature of the particular small-order rings? Finally, within our bank of examples, we found rings where all outliers generate FCSs (like the smallest ring of ternions \cite{HS09}), as well as rings where only some outliers have this property. What distinguishes the two kinds of rings? These are exciting open questions we would like to focus on in the near future.

\section*{Acknowledgement}
The authors wish to thank Andrea Blunk and Asha Rao for comments on early drafts.
JLH gratefully acknowledges the support from the National Scholarship Programme of the Slovak Republic. This work was also partially supported by the VEGA grant agency projects 2/0092/09 and 2/0098/10.

\end{document}